\documentclass[11pt]{article}
\usepackage{hyperref} 

\usepackage{amsmath}
\usepackage{amssymb}
\usepackage{amscd}
\usepackage{textcomp}

\usepackage{amsthm}
\theoremstyle{plain}
\newtheorem{theorem}{Theorem}[section]

\newtheorem{lemma}[theorem]{Lemma}

\newtheorem{proposition}[theorem]{Proposition}

\theoremstyle{definition}
\newtheorem{definition}[theorem]{Definition}

\newtheorem{example}[theorem]{Example}
\newtheorem{remark}[theorem]{Remark}

\theoremstyle{remark}

\usepackage{xy}
\xyoption{all}

\newdir^{ (}{{}*!/-3pt/\dir^{(}}    
\newdir^{  }{{}*!/-3pt/\dir^{}}    
\newdir_{ (}{{}*!/-3pt/\dir_{(}}    
\newdir_{  }{{}*!/-3pt/\dir_{}}

\newcounter{zahl}

\long\def\forget#1{}

\def\theenumi{(\alph{enumi})}

\def\p@enumii{\theenumi}

\newcommand{\DS}{\displaystyle}

\newcommand{\cF}{\mathcal{F}}
\newcommand{\cG}{\mathcal{G}}

\newcommand{\cO}{\mathcal{O}}

\DeclareMathOperator{\Spec}{Spec}

\DeclareMathOperator{\charakt}{char}

\renewcommand{\phi}{\varphi}
\renewcommand{\epsilon}{\varepsilon}

\usepackage{amsfonts}
\newcommand{\BOne} {{\mathchoice{\hbox{\rm1\kern-2.7pt l\kern.9pt}}
                              {\hbox{\rm1\kern-2.7pt l\kern.9pt}}
                              {\hbox{\scriptsize\rm1\kern-2.3pt l\kern.4pt}}
                              {\hbox{\scriptsize\rm1\kern-2.4pt l\kern.5pt}}}}

\newcommand{\BA}{{\mathbb{A}}}

\newcommand{\BC}{{\mathbb{C}}}
\newcommand{\BD}{{\mathbb{D}}}

\newcommand{\BG}{{\mathbb{G}}}

\newcommand{\BP}{{\mathbb{P}}}
\newcommand{\BQ}{{\mathbb{Q}}}

\newcommand{\BZ}{{\mathbb{Z}}}

\newcommand{\CA}{{\cal{A}}}
\newcommand{\CB}{{\cal{B}}}

\newcommand{\CF}{{\cal{F}}}
\newcommand{\CG}{{\cal{G}}}

\newcommand{\CO}{{\cal{O}}}
\newcommand{\CP}{{\cal{P}}}

\newcommand{\CT}{{\cal{T}}}

\newcommand{\FG}{{\mathfrak{G}}}

\let\setminus\smallsetminus

\newcommand{\ol}[1]{{\overline{#1}}}

\newcommand{\wt}[1]{{\widetilde{#1}}}

\usepackage{ifthen}

\newcommand{\invlim}[1][]{\ifthenelse{\equal{#1}{}}
{\DS \lim_{\longleftarrow}}
{\DS \lim_{\underset{#1}{\longleftarrow}}}
}

\newcommand{\dirlim}[1][]{\ifthenelse{\equal{#1}{}}
{\DS \lim_{\longrightarrow}}
{\DS \lim_{\underset{#1}{\longrightarrow}}}
}

\newcommand{\dbl}{{\mathchoice{\mbox{\rm [\hspace{-0.15em}[}}
                              {\mbox{\rm [\hspace{-0.15em}[}}
                              {\mbox{\scriptsize\rm [\hspace{-0.15em}[}}
                              {\mbox{\tiny\rm [\hspace{-0.15em}[}}}}
\newcommand{\dbr}{{\mathchoice{\mbox{\rm ]\hspace{-0.15em}]}}
                              {\mbox{\rm ]\hspace{-0.15em}]}}
                              {\mbox{\scriptsize\rm ]\hspace{-0.15em}]}}
                              {\mbox{\tiny\rm ]\hspace{-0.15em}]}}}}
\newcommand{\dpl}{{\mathchoice{\mbox{\rm (\hspace{-0.15em}(}}
                              {\mbox{\rm (\hspace{-0.15em}(}}
                              {\mbox{\scriptsize\rm (\hspace{-0.15em}(}}
                              {\mbox{\tiny\rm (\hspace{-0.15em}(}}}}
\newcommand{\dpr}{{\mathchoice{\mbox{\rm )\hspace{-0.15em})}}
                              {\mbox{\rm )\hspace{-0.15em})}}
                              {\mbox{\scriptsize\rm )\hspace{-0.15em})}}
                              {\mbox{\tiny\rm )\hspace{-0.15em})}}}}

\newcommand{\dotBD}{\vbox{\hbox{\kern2pt\bf.}\vskip-4.5pt\hbox{$\BD$}}}

\usepackage{color}
\newcounter{commentcounter}

\def\?{\ 
{\bf\color{red}???}\ 
\immediate\write16{}
\immediate\write16{Warning: There was still a question mark . . . }
\immediate\write16{}}

\newbox\mybox
\def\arrover#1{\mathrel{
       \setbox\mybox=\hbox spread 1.4em{\hfil$\scriptstyle#1$\hfil}
       \vbox{\offinterlineskip\copy\mybox
             \hbox to\wd\mybox{\rightarrowfill}}}}

\begin{document}


\author{Esmail\ Arasteh Rad\footnote{This work was completed while the first author was  visiting the Institute for Research in Fundamental Sciences
(IPM).} and Somayeh Habibi\footnote{This research was in part supported by a grant from IPM(No. 93510038).} \footnote{Corresponding author. E-mail address: shabibi@ipm.ir}}

\title{On the Motive of a Fibre Bundle and its Applications}

\maketitle

\begin{abstract}

In this article we compute the motive associated to a cellular fibration $\Gamma$ over a smooth scheme $X$ inside Veovodsky's motivic categories. We implement this result to study the motive associated to a $G$-bundle, and additionally to study motives of varieties admitting a resolution of singularities by a tower of cellular fibrations (e.g. affine Schubert varieties in a twisted affine flag variety).  

\end{abstract}
\textbf{keywords:} mixed Tate motives, cellular fibrations, principal $G$-bundles, wonderful compactification, affine Schubert varieties

\begin{center}

Introduction

\end{center}

A fundamental result of V. Voevodsky states that the (higher) Chow groups of a quasi-projective variety $X$ over a perfect field $k$, of any characteristic, can be viewed as the motivic cohomologies of $X$ \cite{VVa}. As a consequence of this result, A. Huber and B. Kahn established a decomposition of the pure Tate motive $M(Y)$, associated to a smooth variety $Y$, in terms of its motivic fundamental invariants; see Remark \ref{PureTateDec}. Regarding this we prove a motivic version of Leray-Hirsch theorem for a cellular fibration; see Theorem \ref{Leray Hirsch for Voevodsky motives}. 

\begin{theorem}\label{LerayHirschVoevodskyMotives0}
Let $X$ be a smooth irreducible variety over a perfect field $k$. Let $\pi:\Gamma \rightarrow X$ be a proper smooth locally trivial (for Zariski topology) fibration with fiber $F$.  Furthermore assume that $F$ is cellular and satisfies Poincar\'e duality. Then we have a decomposition 
$$
M(\Gamma)\cong \bigoplus_{\geq 0}CH_p(F) \otimes M(X)(p)[2p]
$$

in $DM_{gm}^{eff}(k)$.
  
\end{theorem}
 
We implement the above result, together with the motivic version of the decomposition theorem, due to Corti and Hanamura \cite{C-H} and de Cataldo and Migliorini \cite{C-M}, to study motive of a variety $X$, which admits a resolution of singularities $\tilde{X}$, which in turn can be constructed as a tower of cellular fibrations; see Theorem \ref{toweres}. In particular, one can use this approach to study the motive associated to an \emph{affine Schubert variety} $S(\omega)$ in a \emph{twisted partial affine flag variety}. The affine Schubert varieties were first introduced by G. Pappas and M. Rapoport \cite{P-R} and later received much attention according to the following reasons. First, they play significant role in the theory of local models for Shimura varieties, see \cite{Ha}, \cite{PRS} and \cite{Ri}. Second, they also appear as local models for the moduli stacks of global $G$-shtukas, which are function field analogs for Shimura varieties\footnote{This work, was initially proposed as a part of an extensive research plan, which aims to investigate the motives associated with the moduli stacks of global $G$-shtukas in Voevodsky's motivic categories, using local theory of global $G$-shtukas \cite{AH_Local}, and the techniques which we developed in \cite{Ar-Ha}; see for example proof of \cite[Proposition~4.30]{Ar-Ha}. Notice that, as is proved in \cite[Proposition~2.16]{Var} (also see \cite{AH_Global} for the case of non-constant reductive groups), the moduli stacks of global $G$-shtukas are Deligne-Mumford. Furthermore, the theory of motives for Deligne-Munfored stacks was ``partly'' explored in \cite{Toen} and \cite{Choud}. We leave further discussions in this direction for another occasion.\\}
; see \cite{Var} and also \cite{Ar-Ha}. \\

\noindent
Consequently at the end of Section \ref{Motive of cellular fibrations} we show that 


\begin{proposition}\label{PropMotiveAffSchubertVars}
Let $\BG$ be a connected reductive algebraic group over the local field $K := k\dpl z \dpr$
of Laurent series with algebraically closed residue field $k$. Assume either $\BG$ is constant (i.e. $\BG:=G\times_k k\dpl z\dpr$ for a connected split reductive group $G$ over $k$) or that $\charakt k=0$ and  Grothendieck's standard conjectures and Murre's conjecture hold; see Remark \ref{Corti_Hanamura_Dec}. Then the motive $M(S(\omega))$ is pure Tate.
\end{proposition}


Finally, as the second application of Theorem \ref{LerayHirschVoevodskyMotives0}, we use the combinatorial tools provided by the theory of wonderful compactification according to \cite{DeConciniProcesi} and \cite{Renner}, to study motives of $G$-bundles.

\begin{definition}\label{RMT}
Let $X$ be a scheme over a perfect field $k$.  We denote the smallest tensor thick subcategory of $\textbf{DM}_{gm}^{eff}(k)$(resp. $\textbf{DM}_{gm}^{eff}(k)\otimes \BQ$) containing the category of mixed Tate motives, see Definition \ref{DefMixedTate}, and the motive $M(X)$ (resp. $M(X)\otimes \BQ$)by $\textbf{MT}_X(k,\BZ)$ (resp. $\textbf{MT}_X(k,\BQ)$). We call it the category of relative mixed Tate motives over $X$ with integral (resp. rational) coefficients.
\end{definition}

Let $\CG$ be a $G$-bundle over a smooth irreducible variety $X$ over $k$. 
One might naturally conjecture that the restriction of the motive $M(\CG)$ to $k^{alg}$ lies in the category of relative mixed Tate motives over $X$. The following theorem shows that this expectation is true to a large extent.

\begin{theorem}\label{MGB}
Let $k$ be a perfect field. let $\CG$ be a $G$-bundle over a smooth irreducible variety $X$ over $k$. We have the following statements
\begin{enumerate}
\item\label{alef} 
If $X$ is (geometrically) cellular then $M(\CG)$ is geometrically mixed Tate.
\item\label{ba} 
If $X$ is a smooth projective curve then there is a finite separable field extension $k'/k$ such that the restriction of $M(\CG)$ to $k'$ lies in $\textbf{MT}_X(k',\BQ)$.

\item\label{jim}
If $\CG$ is Zariski locally trivial and $G$ has connected center, then $M(\CG)$ lies in $\textbf{MT}_X(k,\BZ)$.
\item\label{dal}
If $k=\BC$ and $G$ has connected center, then the motive $M(\CG)\otimes \BQ$ lies in $\textbf{MT}_X(k,\BQ)$.
\end{enumerate}
\end{theorem}

\noindent
Note that in the course of the proof of the above theorem \ref{MGB}, we obtain an explicit (nested) filtration on the motive of a $G$-bundle; see Section \ref{NestedFiltration}. The case of torus bundles was previously studied by Huber and Kahn \cite{H-K}. \\
The above discussion in particular implies that the torsor relation $[M(\CG)]=[M(G)].[M(X)]$  holds in $\textbf{K}_0(\textbf{DM}_{gm}^{eff}(k))$; see Proposition \ref{Prop_K_Ring}. Also compare with \cite[Appendix~A]{Beh-Dhi} where they treat the case $\charakt k=0$; see also Remark \ref{torsorrelation}. 

\bigskip

\noindent
The organization of the article is as follows. In section \ref{Notation and Conventions} we fix notation and conventions. In section \ref{Motive of cellular fibrations} we introduce the notion of \textit{motivic (relative) cellular varieties}. We show that for such a variety $X$, the motive $M(X)$ admits a decomposition similar to the Chow motive of a relative cellular variety. Note that we implement these results later to study motives of orbit closures of wonderful compactification of a reductive group; see Proposition \ref{TheoWnderfulCompProperties} and Remark \ref{WhyMotivicCellular}. Furthermore, we establish a motivic version of Leray-Hirsch theorem for cellular fibrations. We use this result to study motives of varieties admitting a resolution of singularities by a tower of cellular fibrations. In particular we prove Proposition \ref{PropMotiveAffSchubertVars}. In Section \ref{SectMotiveofG-bundles} we first recall some results about the geometry of wonderful compactifiation of a reductive group of adjoint type. Subsequently we see that the closure of a $G\times G$-orbit is (motivic) cellular. Using this result and the motivic version of Leray-Hirsch theorem we study the motive associated to a $G$-bundle.  We first consider the case when $\charakt k=0$ and $X$ is geometrically mixed Tate; see Proposition \ref{motive-of-G-bundle1}. Then we consider the case when $\charakt k$ is arbitrary and $X$ is geometrically cellular, proving part \ref{alef} of Theorem \ref{MGB}; see Proposition \ref{motive-of-G-bundle1b}. 
In the last Section \ref{NestedFiltration} we first review the  A.~Huber and B.~Kahn \cite[Section 8]{H-K} (also see Biglari \cite{Big}) filtration on the motive of a \emph{torus bundle}, which they construct using the \emph{theory of slice filtration}. Finally, using the combinatorial tools provided by the theory of wonderful compactification of reductive groups, we establish a nested filtration on the motive of a $G$-bundle $\cG$. Consequently, we prove part \ref{ba}, \ref{jim} and \ref{dal} of Theorem \ref{MGB}.\\

\bigskip

\textit{Acknowledgement.} We are grateful to L. Barbieri Viale and B.Kahn for their
helps and comments on the earlier draft of this article. We thank L. Migliorini for mentioning us an inaccuracy in the proof of Proposition \ref{motive-of-G-bundle1} in a previous version of this article. We also thank C.~V.Anghel, J. Bagalkote and J. Scholbach for editing and helpful discussions regarding this work.

\tableofcontents

\section{Notation and Conventions}\label{Notation and Conventions}

Throughout this article we assume that $k$ is a perfect field. We denote by $\textbf{Sch}_k$ (resp. $\textbf{Sm}_k$) the category of schemes (resp. smooth schemes) of finite type over $k$. 

For $X$ in $\CO b(\textbf{Sch}_k)$, let $CH_i(X)$ and $CH^i(X)$ denote Fulton's $i$-th Chow groups and let $CH_\ast(X):=\oplus_i CH_i(X)$ (resp. $CH^\ast(X):=\oplus_i CH^i(X)$).  



To denote the motivic categories over $k$, such as 

$$
\textbf{DM}_{gm}(k),~ \textbf{DM}_{gm}^{eff}(k),~ \textbf{DM}_{-}^{eff}(k),~ \textbf{DM}_{-}^{eff}(k) \otimes\mathbb{Q},~\text{etc.} 
$$  
and the functors $M:\textbf{Sch}_k\rightarrow \textbf{DM}_{gm}^{eff}(k)$ and $M^c:\textbf{Sch}_k\rightarrow \textbf{DM}_{gm}^{eff}(k)$, constructed by Voevodsky, we use the same notation that was introduced by him in \cite{VV}.\\
For the definition of the geometric motives with compact support in positive characteristic we also refer to \cite[Appendix B]{H-K}.\\

We simply use $A\to B \to C$ to denote a distinguished triangle 
$$
A \to B \to C \to A[1],
$$ 
in either of the above categories. Moreover for any object $M$ of $\textbf{DM}_{gm}(k)$ we denote by $M^*$ the internal Hom-object ${\underline{\text{Hom}}}_{\textbf{DM}_{gm}}(M,\BZ)$.\\

\begin{definition}\label{DefPureTate}
An object of $\textbf{DM}_{gm}(k)$ is called \emph{pure Tate} if it is a (finite) direct sum of copies of $\BZ (p)[2p]$ for $p \in \BZ$.
\end{definition}

\begin{definition}\label{DefMixedTate}
The thick subcategory of $\textbf{DM}_{gm}^{eff}(k)$, generated by $\mathbb{Z}(0)$ and the Tate object $\mathbb{Z}(1)$ is called \emph{the category of mixed Tate motives} and we denote it by $\textbf{MT}(k)$. Any object of $\textbf{MT}(k)$ is called a
 \emph{mixed Tate motive}. A motive $M$ is \emph{geometrically mixed Tate} if it becomes mixed Tate over $k^{alg}$.
\end{definition}

\textbf{CAUTION}: Throughout this article we either assume that $k$ admits resolution of singularities or we consider the motivic categories after passing to coefficients in $\mathbb{Q}$. \\

\noindent
Let us now move to the algebraic group theory side. \\

\noindent
Let $G$ be a connected reductive linear algebraic group over $k$. Suppose that $G$ is split, fix a maximal torus $T$ and a Borel subgroup $B$ that contains $T$.\\ 
We denote by $G^s$ the semi-simple quotient of $G$ and by $G^{ad}$ the adjoint group of $G$. \\
Let $X^\ast(T)$ (resp. $X_\ast(T)$) denote the group of cocharacters (resp. characters) of $G$. 
Let $\Phi:=\Phi(T,G)$ be the associated root system and $\Delta \subseteq \Phi(T,G)$ be a system of simple roots (i.e. a subset of $\Phi$ which form a basis for $Lie(G)$ such that any root $\beta \in \Phi$ can be represented as a sum $\beta =\sum_{\alpha \in \Delta}m_{\alpha} \alpha$, with $m_{\alpha}$ all non-negative or all non-positive integral coefficients). Let $W:=W(T,G)$ and $l:W\rightarrow \mathbb{Z}_+$ denote the corresponding Weyl group and the usual length function on $W$, respectively.  For any subset $I \subseteq \Delta$ we set $\Phi_I$ to be the subset of $\Phi$ spanned by $I$. Furthermore for $u\in W$, $I_u$ will denote the set consisting of those elements of $\Delta$ that do not occur in the shortest expression of $u$. We denote by $W_I$ the subgroup of the Weyl group $W$ generated by the reflections associated with the elements of $\Phi_I$. Let $W^I$ denote a set of representative for $W/W_I$ with minimal length. Recall that any parabolic subgroup $P$ of $G$ is conjugate with a standard parabolic subgroup, i.e. to a group of the form $P_I:=BW_IB$.\\

Let $Y$ be a variety with left $G$-action. To a $G$-bundle $\CG$ on $X$ one associates a fibration $\CG\times^G Y$ with fibre $Y$ over $X$, defined by the following quotient
$$
\CG\times Y\big{\slash}\sim,
$$
where $(x,y)\sim (xg,g^{-1}y)$ for every $g\in G$.

%
%

\section{Motive of cellular fibrations} \label{Motive of cellular fibrations}
\setcounter{equation}{0}

Bellow we introduce the notion of \emph{motivic} relatively cellular varieties. Notice that this notion is slightly weaker than the geometric notion of relatively cellular introduced by Chernousov, Gille, Merkurjev \cite{CGM} and also Karpenco \cite{Kar}.

\begin{definition}\label{relativecellular1}
A scheme $X \in \CO b(\textbf{Sch}_k)$ is called \emph{motivic relatively cellular} (with respect to the functor $M^c(-)$) if it admits a filtration by its closed subschemes:
\begin{equation}\label{filt}
\emptyset= X_{-1} \subset X_0 \subset ... \subset X_n=X 
\end{equation}

together with flat equidimensional morphisms $p_i: U_i:=X_i \setminus  X_{i-1} \rightarrow Y_i$ of relative dimension $d_i$, such that the induced morphisms $p_i^\ast:M^c(Y_i)(d_i)[2d_i] \rightarrow M^c(U_i)$ are isomorphisms in $DM_{gm}^{eff}(k)$. Here $Y_i$ is smooth proper scheme for all $1 \leq i \leq n$. Moreover we say that $X$ is \it{motivic cellular} if all $Y_i$'s appearing in the above filtration are equal to $\Spec k$. 
\end{definition}

Recall that a scheme $X$ is called \textit{relatively cellular} (resp. \textit{cellular}) if it admits a filtration as above \ref{filt} such that $U_i$ is an affine bundle (resp. affine space) over $Y_i$ (resp. over $\Spec k$) via $p_i$. In particular relatively cellular (resp. cellular) implies motivic relatively cellular (resp. motivic cellular).  

\begin{proposition}\label{freechow}
Suppose $k$ admits resolution of singularities. Assume that $X\in \cO b(\textbf{Sch}_k)$ is equidimensional of dimension $n$, which admits a filtration as in the definition \ref{relativecellular1}. Then we have the following decomposition
$$
M^c(X)=\bigoplus_i M^c(Y_i)(d_i)[2d_i].
$$
\end{proposition}

\begin{proof}
We prove by induction on $\dim X$. Consider the following distinguished triangle
$$
M^c(X_{j-1}) \rightarrow M^c(X_j) \xrightarrow{\;g_j} M^c(U_j) \rightarrow M^c(X_{j-1})[1].
$$
The closure of the graph of $p_j : U_j \rightarrow Y_j$ in $X_j \times Y_j$. This defines a cycle in $CH_{dim~X_j}(X_j\times Y_j)$ and since $Y_j$ is smooth this induces the following morphism 
$$
\gamma_j : M^c(Y_j)(d_j)[2d_j] \rightarrow M^c(X_j),
$$
by \cite[Chap. 5, Theorem 4.2.2.3) and Proposition 4.2.3]{VV}, such that $g_j \circ \gamma_j =p_j^\ast$. Thus the above distinguished triangle splits.  Now we conclude by induction hypothesis.
\end{proof}


\begin{remark}\label{relativecellular2}
One can define a variant of definition \ref{relativecellular1} with respect to the functor $M(-)$. Notice that in this set up one has to replace $p_i^\ast$ by $p_{i_{\ast}}$. Note further that in this case  one may drop the assumption that $p_i$ is flat. Accordingly, a variant of the Proposition \ref{freechow} holds after imposing some additional condition. Namely, to apply Gysin triangle \cite[page 197]{VV} we have to assume that all $X_i$s appearing in the filtration of $X$ are smooth; compare proof of Proposition \ref{freechow}). Note further that in this case it is not necessary to assume that $k$ admits resolution of singularities.

\end{remark}

\begin{remark}\label{puretate}
Assume that $X$ is a motivic relatively cellular scheme, such that $Y_i$ is pure Tate for every $1 \leq i \leq n$. Then using noetherian induction and Gysin triangle one can show that $X$ is pure Tate. 
\end{remark}
Let $\textbf{Ab}$ be the category of abelian groups. Let us recall that there is a fully faithful tensor triangulated functor 
$$
i: \textbf{D}_f^b(\textbf{Ab}) \rightarrow \textbf{DM}_{gm}^{eff}(k),
$$
where  $\textbf{D}_f^b(\textbf{Ab})$ is the full subcategory of the bounded drived category $\textbf{D}^b(\textbf{Ab})$, consisting of objects with finitely generated cohomology groups; see \cite[Proposition 4.5.]{H-K}.\\

Let us state the following proposition which is proved by B.~ Kahn \cite[Proposition 3.4]{BKahn}.

\begin{proposition}\label{MotivewcsofCellularVariety}
Assume that $k$ admits resolution of singularities. For a cellular variety $X \in \cO b(\textbf{Sch}_k)$, there is a canonical isomorphism $$
\bigoplus_{p\geq0} CH_p(X) \otimes \mathbb{Z}(p)[2p] \rightarrow M^c(X),
$$ 

which is functorial both with respect to proper or flat morphisms.
\end{proposition}

\begin{remark}\label{PureTateDec}
We recall the following variant of the above decomposition for the motive $M(X)$.
Assume that $X \in \cO b(\textbf{Sm}_k)$ and assume further that it is equidimensional. Then  there is a natural isomorphism 
$$
\bigoplus_{p \geq 0}CH^p(X)^{\vee} \otimes \mathbb{Z}(p)[2p] \tilde{\to} M(X)
$$
in $\textbf{DM}_{gm}^{eff}(k)$, according to \cite[Corollary~3.5]{BKahn}. Here $CH^p(X)^{\vee}$ denote the dual $\mathbb{Z}$-module.\\ 
More generally in \cite[Proposition 4.10]{H-K} Huber and Kahn prove that for a smooth variety $X$ over $k$, if the associated motive $M(X)$ is pure Tate then there is a natural decomposition $M(X) \cong \bigoplus_p c_p(X)(p)[2p]$ in terms of the corresponding motivic fundamental invariants $c_p(X)$.
\end{remark}

\begin{remark}\label{motive-general-flag-var}
Let $G$ be a split reductive group over a perfect field $k$, and let $P$ be a parabolic subgroup of $G$ which is conjugate with a standard parabolic subgroup $P_I$. Then one has the isomorphism $M(G/P)\cong \bigoplus_{w\in W^I} \mathbb{Z}(l(\omega))[2l(w)]$. Namely, the decomposition $G=\coprod_{w\in W^I} BwP$, induces a cell decomposition $G/P\cong G/P_I=\coprod_{w\in W^I} X_w$, where $X_w$ is isomorphic to the affine space $\BA^{l(w)}$. Note that the cycles $[\overline{X_w}]$ form a set of generators for the free module $CH_\ast(G/P)$.  
\end{remark}

We now want to compute the motive associated to a fiber bundle. Recall that the naive version of \emph{Leray-Hirsch theorem} does not even hold for the Chow functor. One way to tackle the problem in the algebraic set-up is to impose some stronger conditions on the fiber. Namely, one should assume that the fiber admits cell decomposition and further satisfies Poincar\'e duality (i.e. the degree map $CH_0(F)\to \mathbb{Z}$ is an isomorphism and the intersection pairings $CH_p(F)\otimes CH^p(F)\rightarrow CH_0(F)$ are perfect parings).\\ 
Let $f: \Gamma \rightarrow X$ be a smooth proper morphism that is locally trivial for Zariski topology, with fiber $F$ satisfying the above conditions. Let $\zeta_1,...,\zeta_m$ be homogeneous elements of $CH^\ast(\Gamma)$ whose restriction to any fiber form a basis of its Chow group over $\mathbb{Z}$. Then the Leray-Hirsch theorem for Chow groups states that the homomorphism
$$
\varphi : \oplus_{i=1}^m CH_\ast(X) \rightarrow CH_\ast({\Gamma})~~~~~~,~\varphi(\oplus \alpha_i)= \Sigma \zeta_i \cap f^\ast\alpha_i
$$
 is an isomorphism. When $X$ is non-singular, it means that $\zeta_i$ form a free basis for $CH^\ast(\Gamma)$ as a $CH^\ast(X)-module$. For the proof, we refer the reader to \cite[appendix C]{Col-Ful}.\\
 
Let us now state the \emph{motivic version of the Leray-Hirsch theorem}.
 
\begin{theorem}\label{Leray Hirsch for Voevodsky motives}
Let $X$ be a smooth irreducible variety over $k$. Let $\pi:\Gamma \rightarrow X$ be a proper smooth locally trivial (for Zariski topology) fibration with fiber $F$.  Furthermore, assume that $F$ is cellular and satisfies Poincar\'e duality. Then one has an isomorphism
$$
M(\Gamma)\cong \bigoplus_{p \geq 0}CH_p(F) \otimes M(X)(p)[2p]
$$  
in $\textbf{DM}_{gm}^{eff}(k)$.
\end{theorem}

\begin{proof}
Take a set of homogeneous elements $\{\zeta_{i,p}\}_{i,p}$ of $CH^\ast(\Gamma)$ such that for any $p$ the restrictions of $\{\zeta_{i,p}\}_i$ to any fiber $\Gamma_x\cong F$ form a basis for $CH^p(\Gamma_x)$. Notice that since $X$ is irreducible, it is enough that the restrictions of the $\zeta_{i,p}$'s generate $CH_\ast
(\Gamma_x)$ for the fiber over a particular $x$. 

By \cite{MVW} Theorem 14.16 and Theorem 19.1, for each $i$, $\zeta_{i,p}$ defines a morphism $M(\Gamma) \rightarrow \mathbb{Z}(p)[2p]$. Summing up all these morphisms and taking dual, by Poincar\'e duality we get the following morphism  
$$
\varphi: M(\Gamma)\rightarrow\bigoplus_p CH_p(F)\otimes \mathbb{Z}(p)[2p].
$$

Composing $M(\Delta): M(\Gamma)\rightarrow M(\Gamma\times \Gamma)\cong M(\Gamma) \otimes M(\Gamma)$, which is induced by the diagonal map $\Delta: \Gamma\times \Gamma \rightarrow \Gamma$, with $M(\pi)\otimes \varphi$, we obtain
$$
M(\Gamma)\rightarrow \bigoplus_p CH_p(F)\otimes M(X)(p)[2p].$$

Now take a covering $\{U_i\}$ of $X$ that trivializes $\Gamma$.
The restriction of this global morphism to $U_j$ is induced by the restriction of $\zeta_{i,p}$s to $U_j$. The same holds over intersections, i.e. these morphisms fit together when we pass to $U_j \cap U_k$. Thus by Mayer-Vietoris triangle (see \cite[Chapter 5, (4.1.1)]{VV}) we may reduce to the case that $\Gamma$ is a trivial fibration $X \times_k F$. This precisely follows from K\"unneth formula \cite[Proposition 4.1.7]{VV} and Remark \ref{PureTateDec} above.

\end{proof}

Let us now state the first application of the above theorem which we propose in this article. Let $\tilde{X}$ be a variety over a perfect field $k$. Suppose that $\tilde{X}$ sits in a tower

$$
\CD
\tilde{X}_n:=\tilde{X}\\
@V{p_{n-1}}VV\\
\tilde{X}_{n-1}\\
\vdots\\
@VVV\\
\tilde{X}_0
\endCD
$$

such that $\tilde{X}_i\rightarrow \tilde{X}_{i-1}$ is a proper smooth locally trivial fibration with fibre $F_i$. Furthermore assume that $F_i$ is cellular and satisfies Poincar\'e duality. We call such a variety a \emph{tower variety} over $\tilde{X}_0$. \\

\begin{theorem}\label{toweres}
Let $f: \tilde{X} \rightarrow X$ be a surjective semismall morphism. Further assume that $\tilde{X}$ is a tower variety over a smooth proper scheme $\tilde{X}_0$. Then the motive $M(X)$ is a summand of 
$$
\bigotimes_{i=0}^{n-1} \left(\bigoplus_{p \geqslant 0}CH_p(F_i)\otimes \BZ(p)[2p]\right) \otimes M(\tilde{X}_0)
$$
in $\textbf{DM}_{gm}(k)\otimes \BQ$.
\end{theorem}

\begin{proof}
This follows from the motivic version of the \emph{decomposition theorem} \cite[Theorem. 2.3.8]{C-M}, Theorem \ref{Leray Hirsch for Voevodsky motives} and the embedding theorem \cite[Chap. 5, Proposition 2.1.4]{VV}.
\end{proof}

\begin{remark}\label{Corti_Hanamura_Dec}
Assuming Grothendieck's standard conjectures and Murre's conjecture, see \cite[Paragraph 3.6]{C-H}, Corti and Hanamura
prove that the decomposition theorem holds in the category of relative Chow motives $CHM_{S,\BQ}$ with rational coefficients. In this case one may drop the semismallness from the hypotheses of the above corollary. Notice that when $S$ is proper, then $CHM_{S,\BQ}$ maps to rational Chow motives.
\end{remark}

Let $\BG$ be a connected reductive algebraic group over the local field $K := k((z))$
of Laurent series with algebraically closed residue field $k$ and ring of integers $\CO_K = k\dbl z \dbr$. Recall that to such a group $\BG$ over a local field $K$, Bruhat and Tits \cite{B-T} associate a 
building $\CB(\BG)$.  Moreover any maximal split torus $S$ defines an apartment $\CA := \CA(\BG, S)$ which is called the reduced apartment of $\BG$ associated with $S$. For any $F \in \CA(\BG,S)$ we let $P_F$ denote the corresponding \emph{parahoric} group scheme and $\CF\ell_{P_F}$ the associated \emph{twisted affine flag variety}. For an element $\omega$ of the Iwahori-Weyl group $\wt{W}$, we let $S(\omega)$ denote the associated affine Schubert variety in $\CF\ell_{P_F}$; see \cite[Sec. 8]{P-R}.\\


\begin{proof}{of Proposition \ref{PropMotiveAffSchubertVars}.}
The source of the Bott-Samelson-Demazure resolution $\Sigma(\omega) \to S(\omega)$, constructed in \cite[Corollary~3.5]{Ri}, is an iterated extension of projective homogeneous varieties \cite[Remark~2.9]{Ri}. Therefore by Theorem \ref{Leray Hirsch for Voevodsky motives} the motive $M(\Sigma(\omega))$ is pure Tate. Note that generalized flag varieties satisfy Poincar\'e duality; see \cite{Ko}. Now the proposition follows from Ng\^o-Polo\cite[Lemma 9.3]{Ngo-Polo}, Theorem \ref{toweres} and Remark \ref{Corti_Hanamura_Dec}.  Notice that the affine Schubert variety $S(\omega)$ is projective and therefore the splitting of the corresponding motive in the category of relative Chow motives, gives the splitting in the category of rational Chow motives and hence in $\textbf{DM}_{gm}^{eff}(k)$ by \cite[Chap. 5, Proposition 2.1.4]{VV}. 
\end{proof}

\begin{remark}
One can explicitly compute the summands that may appear in the decomposition of the motive associated with $S(\omega)$ according to Remark~\ref{PureTateDec},  \cite[Theorem~2.3.8]{C-M}, Remark \ref{Corti_Hanamura_Dec} and finally \cite[Remark~2.9]{Ri}. 
\end{remark}
%
%

\section{Motive of a G-bundle} \label{SectMotiveofG-bundles}

In this section we study motives of G-bundles over a base scheme which is (geometrically) cellular, or
more generally, over a base scheme which is (geometrically) mixed Tate. Let us first recall the following result of A. Huber and B. Kahn \cite[Proposition 5.3]{H-K}.\\

\begin{proposition}\label{geometricmixedtate}

An object $M \in \textbf{DM}_{gm}^{eff}(k)$ is geometrically mixed Tate if and only if there is a finite
separable extension $E$ of $k$ such that the restriction of $M$ to $\textbf{DM}_{gm}^{eff}(E)$
is mixed Tate.

\end{proposition}

\begin{definition} \label{DefMixedArrangement}
Let $X \in \cO b(\textbf{Sch}_k)$. We say that $X$ is \emph{mixed Tate} if the associated motive $M^c(X)$ lies in the subcategory $\textbf{MT}(k)$ of mixed Tate motives. Let $\{ X_i\}_{i=1}^n$ be the set of irreducible components of  $X$. We call $X$ a \textit{configuration of mixed Tate varieties} if
\begin{enumerate}
\item[i)] $X_i$ is mixed Tate for $1 \leq i \leq n$, and
\item[ii)] the union of the elements of any arbitrary subset of $\{X_{ij} := X_i \cap X_j\}_{i \neq j}$ is either a configuration of mixed Tate varieties or empty.
\end{enumerate} 
\end{definition}

\begin{lemma}\label{LemMixedArrangement}
The motive of a configuration of mixed Tate varieties is mixed Tate.
\end{lemma}
\begin{proof}
We prove by induction on the dimension $r$ of the mixed Tate configuration. The statement is obvious for $r=0$. Suppose that the lemma holds for all mixed Tate configurations of dimension $r<m$. Let $X=X_1 \cup \dots \cup X_n$ be a configuration of mixed Tate varieties of dimension $m$, here $X_i$ denote an irreducible component of $X$. For inclusion $\bigcup_{i \neq j}X_{ij} \subset \bigcup_{i=1}^n X_i$, we have the following induced localization distinguished triangle 
$$
M^c(\bigcup_{i \neq j}X_{ij})\rightarrow M^c(X_1 \cup...\cup X_n) \rightarrow M^c(\bigcup_{i=1}^n X_i \setminus\bigcup_{i \neq j}X_{ij}).
$$
By the induction assumption, $ M^c(\bigcup_{i \neq j}X_{ij})$ is mixed Tate. On the other hand we have:
$$
M^c(\bigcup_{i=1}^n (X_i\setminus\bigcup_{i \neq j}X_{ij}))= \bigoplus_{i=1}^n M^c(X_i\setminus\bigcup_{i \neq j}X_{ij}).
$$
It only remains to show that for every $i$, $M^c(X_i\setminus\bigcup_{i \neq j}X_{ij})$ is mixed Tate. To see this, consider the following distinguished triangle
$$
M^c(\bigcup_{i \neq j}X_{ij}) \rightarrow M^c(X_i) \rightarrow M^c(X_i\setminus\bigcup_{i \neq j}X_{ij}).
$$
Notice that $M^c(\bigcup_{i \neq j}X_{ij})$ is mixed Tate by induction hypothesis. 
\end{proof}

In the sequel, we use the theory of \emph{wonderful compactification} of semi-simple algebraic groups of adjoint type to relate motive of a $G$-bundle to the motives associated to certain cellular fiber bundles.\\

\begin{remark} \emph{Wonderful Compactification} In \cite{Con-Pro} De Concini and Procesi have introduced the wonderful compactification of a symmetric space. In particular their method produces a smooth canonic compactification $\overline{G}$ of an algebraic group $G$ of adjoint type. Note that in \cite{Con-Pro} they study the case that the group $G$ is defined over $\mathbb{C}$. Nevertheless most of the theory carries over for any algebraically closed field of arbitrary characteristic. However there are some subtleties in positive characteristic which we mention bellow, see Proposition \ref{TheoWnderfulCompProperties}\ref{bij}.
\newline 
As a feature of this compactification, there is a natural $G \times G$-action on $\overline{G}$, and the arrangement of the orbits can be explained by the associated weight polytope. Let us briefly recall some facts about the construction of $\overline{G}$ and the geometry of its $G \times G$-orbits and their closure.

Let $\rho_{\lambda}:G \rightarrow GL(V_{\lambda})$ be an irreducible faithful representation of $G$ with strictly dominant highest weight $\lambda$. One defines the compactification $X_{\lambda}$ of $G$ as the closure $\overline{\mathbb{P}(\rho_{\lambda}(G))}$ (in $\BP(End(V_\lambda))$) of the projectivization $\mathbb{P}(\rho_{\lambda}(G))$ of $\rho_{\lambda}(G)$.\\
It is proved in \cite{DeConciniProcesi} that when $G$ is of adjoint type, $X_{\lambda}$ is smooth and independent of the choice of the highest weight. In this case the resulting compactification is called \emph{wonderful compactification} and we denote it by $\overline{G}$.  
\end{remark}

The following proposition explains the geometry of the wonderful compactification of $G$ and the closures of its $G \times G$-orbits. Furthermore, it provides an effective method to compute their cohomologies.\\ 
Before stating this proposition, let us fix the following notation. Consider the correspondence between polytopes and fans, which associates to a polytope its normal fan. We let $\CP_C$ denote the polytope associated to the fan of Weyl chamber and its faces.

\begin{proposition}\label{TheoWnderfulCompProperties} 
Keep the above notation, we have the following statements:

\begin{enumerate}
\item\label{polytop}
There is a one-to-one correspondence between the $G \times G$-orbits of $\overline{G}$ and the orbits of the action of the Weyl group on the faces of the polytope $\CP_C$, which preserves the incidence relation among orbits (i.e. for two faces $\cF_1$ $\cF_2$ of the polytope $\CP_C$, if $\cF_1 \subseteq \cF_2$ then the orbit corresponding to the face $\cF_1$ is contained in the closure of the orbit corresponding to $\cF_2$).   

\item\label{bij} 
Let $I \subset \Delta$ and $\mathcal{F}=\mathcal{F}_I$ the associated face of $\CP_C$. Let $D_{\mathcal{F}}$ denote the closure of the orbit corresponding to the face $\mathcal{F}$. Then $D_{\mathcal{F}}=\sqcup_{\alpha\in W\times W}C_{\mathcal{F},\alpha}$, such that for each $\alpha:=(u,v)$ there is a bijective morphism
\begin{equation}\label{bijection}
\mathbb{A}^{n_{{\mathcal{F}},\alpha}}\rightarrow C_{\mathcal{F},\alpha}
\end{equation}
where $n_{\mathcal{F},\alpha}=l(w_0)-l(u)+\mid I\cap I_u\mid +l(v)$ and $w_0$ denotes the longest element of the Weyl group. In particular when $\charakt k=0$ (resp. $\charakt k> 0$) $D_{\mathcal{F}}$ is cellular (resp. motivic cellular).
\item\label{normalcrossing}
$\overline{G} \setminus G$ is a normal crossing divisor, and its irreducible components form a mixed Tate configuration.   
\end{enumerate}
\end{proposition} 

\begin{proof}
For the proof of \ref{polytop} we refer to \cite[Proposition 8]{Tima}. The existence of the decomposition $D_{\mathcal{F}}=\sqcup_{\alpha\in W\times W}C_{\mathcal{F},\alpha}$ and bijective morphism \ref{bijection} is the main result of Renner in \cite{Renner}. The fact that $D_{\cF}$ is cellular in characteristic zero follows from Zariski main theorem. Finally the assertion that $D_{\cF}$ is motivic cellular in positive characteristic follows from the fact that any universal topological homeomorphism induces isomorphism of the associated h-sheaves, see \cite[Proposition 3.2.5]{VVII}. Finally \ref{normalcrossing} follows from \ref{polytop}, \ref{bij} and Remark \ref{puretate}.
\end{proof}

\begin{remark}\label{WhyMotivicCellular}
When $\charakt k > 0$ then we don't know whether the bijection \ref{bijection} is an isomorphism or not and thus $D_\cF$ might not fit the definition \ref{MotivewcsofCellularVariety} of cellular varieties. Nevertheless, as we have seen above, the variety $D_\CF$ is \emph{motivic cellular} and hence enjoys a similar decomposition; see Proposition \ref{freechow} and remark \ref{relativecellular2}. 
\end{remark}

\begin{proposition}\label{motive-of-G-bundle1}
Assume that $\charakt k=0$. Let $G$ be a connected reductive group over $k$ with connected center $Z(G)$. Let $\cG$ be a $G$-bundle over an irreducible variety $X\in \CO b(\textbf{Sm}_k)$. Suppose that $\CG$ is locally trivial for Zariski topology and $X$ is geometrically mixed Tate, then $M(\cG)$ is also geometrically mixed Tate.
\end{proposition}

\begin{proof}
We may assume that the base field $k$ is algebrically closed.\\
Let us first assume that $G$ is a semisimple group of adjoint type. 
 Then $G$ admits a wonderful compactification $\overline{G}$ which is smooth. By construction, there is $(G \times G)$-action on $\overline{G}$. Consider the $\overline{G}$-fibration  $\overline{\mathcal{G}}:= \CG \times ^G\overline{G}$ over $X$ (here $G$ acts on $\overline{G}$ via the embedding $G\hookrightarrow G\times e\subseteq G\times_k G$). Consider the following generalized Gysin distinguished triangle 
$$
M(\mathcal{G}) \rightarrow M(\overline{\mathcal{G}}) \rightarrow M^c(\overline{\mathcal{G}}  \setminus \mathcal{G})^\ast(n)[2n] \rightarrow M(\mathcal{G})[1]
$$
\cite[page 197]{VV}, corresponding to the open immersion $\mathcal{G} \hookrightarrow \overline{\mathcal{G}}$. Here $n:= \dim G+ \dim X$.\\
By Proposition \ref{TheoWnderfulCompProperties}, $\overline{G}$ admits a cell decomposition and therefore by Theorem \ref{Leray Hirsch for Voevodsky motives}, $M^c(\overline{\mathcal{G}})$ is mixed Tate. Hence it suffices to show that $M^c({\overline{\mathcal{G}} \setminus \mathcal{G}})$ is mixed Tate.\\
Let us now look at the geometry of the closures of $(G \times G)$-orbits. As we mentioned in Proposition \ref{TheoWnderfulCompProperties} a), these orbit closures are indexed by a subset of  faces of Weyl chamber, in such a way that the incidence relation between faces gets preserved. Note that by Proposition \ref{TheoWnderfulCompProperties} b) the closure of these orbits also admit a cell decomposition. Thus by Theorem \ref{Leray Hirsch for Voevodsky motives} the irreducible components of $\overline{\mathcal{G}}\setminus \mathcal{G}$ form a mixed Tate configuration. Now Lemma \ref{LemMixedArrangement} implies that $M^c(\overline{\mathcal{G}}\setminus \mathcal{G})$ is mixed Tate. 
\newline
Now let us assume that $G$ is a reductive with connected center $Z:=Z(G)$. Let $\mathcal{G}'$ denote the $G^{ad}$-bundle associated with $\CG$. As we have shown above the motive $M(\mathcal{G}')$ is mixed Tate.  
Notice that any torus bundle is locally trivial for Zariski topology by Hilbert's theorem 90. Take a toric compactification $\overline{Z}$ of $Z$ and embed $\mathcal{G}$ into $\mathcal{Z} := \mathcal{G} \times^Z \overline{Z}$, which is a toric fibration over $\mathcal{G}'$. The irreducible components of the complement of $\mathcal{G}$ in $\mathcal{Z}$ are toric fibrations over $\mathcal{G}'$. Since fibers are toric (and hence cellular) and $M(\mathcal{G}')$ is mixed Tate, by Theorem \ref{Leray Hirsch for Voevodsky motives} we may argue that these irreducible components form a mixed Tate configuration and hence we may conclude as above.
\end{proof}


In the statement of the Proposition\ref{motive-of-G-bundle1}, the assumption that the $G$-bundle $\CG$ is locally trivial for Zariski topology is restrictive; nevertheless, as we will see below,  when $X$ is geometrically cellular, it turns out that this assumption is not necessary. Before proving this, let us state the following lemma.

\begin{lemma}\label{ReuctiveGroup}
Let $G$ be a connected reductive group over $k$, then the motive associated to $G$ is geometrically mixed Tate.
Furthermore if $G$ is a split reductive group then $M(G)$ is mixed Tate.
\end{lemma}
\bigskip
\begin{proof}
For the first statement we may assume that $k$ is algebraically closed. Let $T$ be a maximal split torus in $G$ of rank $r$. We view $G$ as a $T$-bundle over $G/T$ under the projection $\pi:G \rightarrow G/T$, and $\BP^r_k$ as a compactification of $T$. Let $\overline{\mathcal{T}} := G \times ^{T}\BP^r_k $ be the associated projective bundle over $G/T$.  By projective bundle formula \cite[Theorem 15.12]{MVW} we have
$$
M(\overline{\mathcal{T}})= M(\BP^r_k ) \otimes M(G/T).
$$
On the other hand $B = T \ltimes U$, where $B$ is a Borel subgroup of $G$ containing $T$ and $U$ is the unipotent  part of $B$. Notice that, as a variety, $U$ is isomorphic to an affine space over $k$. Since the fibration $G \rightarrow G/B$ is the composition of $G \rightarrow G/T$ and $U$-fibration $G/T \rightarrow G/B $, we deduce by Corollary \ref{motive-general-flag-var} that $M(\overline{\mathcal{T}})$ is pure Tate. As in the last paragraph of the proof of Proposition \ref{motive-of-G-bundle1}, one can embed $G$ into $\overline{\mathcal{T}}$ over $G/T$ and verify that the irreducible components of its complement form a mixed Tate configuration. The second part of the lemma is similar, only since $G$ is split, one does not need to pass to an algebraic closure.
\end{proof}

\begin{remark}
In \cite[Section 8]{H-K} A.~Huber and B.Kahn proved that the motive of a split reductive group is mixed Tate. Their proof relies on the filtration on the motive of a \emph{torus bundle}. They produce this filtration using their theory of slice filtration; see also section \ref{NestedFiltration}. One can alternatively use the explicit filtration we establish in section \ref{NestedFiltration}; see also Lemma \ref{NestedFilt} Example \ref{TorusBundle}.
\end{remark}

\begin{proposition}\label{motive-of-G-bundle1b}
Let $G$ be a connected reductive group over $k$. Let $\cG$ be a $G$-bundle over an irreducible variety $X\in \CO b(\textbf{Sm}_k)$. Suppose in addition that $X$ is geometrically cellular. Then $M(\cG)$ is geometrically mixed Tate.
\end{proposition}
\begin{proof}
We may assume that $k$ is algebraically closed. Let 
$$
\emptyset= X_{-1} \subset X_0 \subset ... \subset X_n=X 
$$
be a cell decomposition for $X$, where $U_i:=X_i \setminus X_{i-1}$ is isomorphic to $\mathbb{A}^{d_i}_k$.
We prove by induction on $n$. Consider the following Gysin distinguished triangle
\begin{equation}\label{FiltOverCellX}
M(\cG|_{U_n}) \rightarrow M(\cG) \rightarrow M(\cG|_{X_{n-1}})(c)[2c],
\end{equation}
where $c$ is the codimension of $X_{n-1}$ in $X$. By Raghunathan's Theorem \cite{Raghunathan2} (the generalization of the well-known conjecture of Serre about triviality of vector bundles over an affine space), the restriction of $\cG$ to $U_n$ is trivial. Therefore $M(\CG|_{U_n})$ is mixed Tate by Lemma \ref{ReuctiveGroup} and K\"unneth theorem \cite[Proposition 4.1.7]{VV}. On the other hand $M(\cG|{X_{n-1}})$ is mixed Tate by induction hypothesis.
\end{proof}

\section{Filtration on the motive of a G-bundle}\label{NestedFiltration}

\bigskip

Recall that in section \ref{SectMotiveofG-bundles} we studied the motive associated to a $G$-bundle over a base scheme $X$ whose motive $M(X)$ is geometrically mixed Tate. In the sequel, we produce a nested filtration (in terms of incidence relations between faces of a convex body) on the motive of a $G$-bundle over a general base scheme.\\ 
Let us first recall the following filtration on the motive of a torus bundle in $\textbf{DM}_{gm}^{eff}(k)$, constructed by A. Huber and B. Kahn; see \cite{H-K} where they produce this filtration as an application of the theory of \emph{slice filtration} and further use it to study motive of a split reductive group \cite[Lem. 9.1]{H-K}. Let us briefly recall their construction. \\

Let $T$ be a split torus of rank $r$ and let $\mathcal{T}$ be a $T$-bundle over a scheme $X \in \CO b(\textbf{Sm}_k)$. Let $\Xi:= Hom(\mathbb{G}_m , T)$ denote the cocharacter group. Then one has the following diagram of distinguished triangles in $\textbf{DM}_{gm}^{eff}(k)$ 
\bigskip

\begin{equation}\label{Huber_Kahn_Filt}
\xymatrix @C=0pc {
\nu^{\geqslant p+1}_X M(\CT) \ar[rrrrrr]& & & & & & \nu^{\geqslant p}_X M(\CT) \ar[dd]& && \nu ^{\geqslant2}_X M(\CT) \ar[rrrrr]& & & && \nu^{\geqslant 1}_X M(\CT) \ar[rrrrrr]\ar[dd]& & && & & M(\CT)\ar[dd]  \\
& & & & &[1] & & & & ... & &  & & [1] & && & & & [1] & \\
 & & & & & & \lambda_p(X,T)\ar[uullllll]& & & &  & & & & \lambda_1(X,T)\ar[uulllll] & & & & & & \lambda_0(X,T)\ar[uullllll]
}
\end{equation}

\bigskip

where $\lambda_p(X,T):= M(X)(p)[p] \otimes \Lambda^p(\Xi)$ for $0 \leq  p \leq  r$. Note that $M(\CT) \cong \nu_X^{\geqslant 0}M(\CT)$ and $\nu_X^{\geqslant r+1}M(\CT)=0$. For the details on the construction of relative slice filters $\nu_X^{\geqslant p}M(\mathcal{T})$ see \cite[Sec. 8]{H-K}.

\bigskip

Now, following the method which we introduced in section \ref{SectMotiveofG-bundles}, we want to explain our construction of a nested filtration on the motive of a $G$-bundle.

\bigskip

\noindent
Let $\cG$ be a $G$-bundle over $X$, where $G$  is a linear algebraic group. Let $\tilde{G}$ be a compactification of $G$. Suppose that the irreducible components of $D:= \tilde{G} \setminus G$ form a mixed Tate configuration $D=\cup_{i=1}^m D_i$, such that $D^J:= \cap_{i \in J} D_i$ is either irreducible or empty  for any $J \subseteq \{ 1,\dots,m\}$. We assume that there exist a polytope whose faces correspond to those subsets $J$ of $\{1,...,m\}$ such that $D^J$ is non empty (with face relation $\CF_2$ is a face of $\CF_1$ if we have the inclusion $J_2 \subseteq J_1$
of the corresponding sets). Let $\CP$ be the dual of this polytope. For each face $\cF$ of $\CP$, we denote by $D_{\cF}$ the associated subvariety of $D$, regarding the above correspondence. We set $D_\CP :=\ol G$. For each $1 \leq r \leq m$, let $Q_r$ be the set consisting of all faces in $\CP$ of codimension $r$. Let $\partial \cF$ denote the boundary of $\CF$, i.e. the set $\{\cF \cap \mathcal{F}' | \mathcal{F}' \in Q_1\}\setminus\{\cF\}$.\\
Let $\tilde{\cG}$ denote  the compactification $\cG \times^G \tilde{G}$ of $\cG$ and let $\mathcal{D}_{\cF}$ be the associated $D_{\cF}$-fibration over $X$.  Furthermore set $\mathcal{D}_\CP :=\ol \CG$. According to the above discussion, one may easily derive the following filtration on the motive of a $G$-bundle.

\begin{lemma}\label{NestedFilt}
There is a nested filtration on $M^c(\cG)$ by distinguished triangles, indexed by codimension $r$ and faces $\cF\in Q_r$

\begin{equation}\label{Filt1}
\CD
M^c(\bigcup_{ \cF \in Q_1} \mathcal{D}_{\cF}=\tilde{\cG}\setminus \cG) \rightarrow M^c(\mathcal{D}_\CP=\tilde{\cG}) \rightarrow M^c(\mathcal{D}_\CP\setminus\bigcup_{ \cF \in Q_1} \mathcal{D}_{\cF}=\cG)\\
\vdots\\
M^c (\bigcup_{ \cF \in Q_{r+1}} \mathcal{D}_{\cF}) \rightarrow M^c (\bigcup_{\cF \in Q_{r}} \mathcal{D}_{\cF}) \rightarrow \oplus_{\cF \in Q_r} M^c(\mathcal{D}_{\cF} \setminus \bigcup_{\cF' \in \partial \cF} \mathcal{D}_{\cF'}),\\
\vdots\\
\endCD
\end{equation}

and for each $\cF\in Q_r$ the triangle
$$
\CD
M^c (\bigcup_{\cF' \in \partial \cF} \mathcal{D}_{\cF'}) \rightarrow M^c ( \mathcal{D}_{\cF}) \rightarrow  M^c(\mathcal{D}_{\cF} \setminus \bigcup_{\cF' \in \partial \cF} \mathcal{D}_{\cF'}),\\
\endCD
$$
\noindent
is the first line of a nested filtration obtained by replacing $\CP$ by $\cF$.\\

\end{lemma}

\begin{proof}
It proceeds similar to the proof of Lemma \ref{LemMixedArrangement}.
\end{proof}
Note that the above filtration is particularly interesting when $\mathcal{D}_{\mathcal{F}}$ is a cellular fibration. In this situation we may apply Theorem \ref{Leray Hirsch for Voevodsky motives} to compute $M^c(\mathcal{D}_{\cF})$. Let us consider the following two cases.

\begin{example}\label{TorusBundle}
Let $T$ be a split torus of rank $r$ as above and $\mathcal{T}$ be a $T$-bundle over $X$. This is locally trivial for Zariski topology on $X$ by Hilbert's theorem 90. Consider the projective space $\mathbb{P}^r$ as a toric compactification of $T$ corresponding to the standard r-simplex $\Delta^r$. So we have $\CP=\Delta^r$. Note that in this case for each face $\cF \in \Delta^r$, $\mathcal{D}_{\cF}$ is in fact a projective bundle over $X$. Hence one may use the projective bundle formula \cite[Theorem~15.11]{MVW} to compute $M^c(\mathcal{D}_{\cF})$. In particular when $M^c(X)$ is mixed Tate, one may prove recursively that $M^c(\mathcal{T})$ is mixed Tate.
\end{example}

\begin{example}\label{ReductiveBundle}
Let $G$ be a semi-simple group of adjoint type and $\tilde{G}:=\overline{G}$ its wonderful compactification. In this case the polytope $\CP$ coincide with the one in Proposition \ref{TheoWnderfulCompProperties}. Note that for each face $\cF$ of $\CP$, $D_{\cF}$ admits a cell decomposition, see Proposition \ref{TheoWnderfulCompProperties}. Let us mention that for a regular compactification of a reductive group $G$, each closed orbit $D_{\cF}$ corresponding to a vertex $\cF$ is isomorphic to product of flag varieties $G/B \times G/B$. In particular $\mathcal{D}_{\cF}$ is a cellular fiberation, see \cite{Brion} for details.
\end{example}

Recall that we phrased part \ref{alef} of Theorem \ref{MGB} in Proposition \ref{motive-of-G-bundle1b}. We
now want to complete the proof of This theorem.

\begin{proof} of parts \ref{ba}, \ref{jim} and \ref{dal} of Theorem \ref{MGB}.\\
\noindent
By the well-known theorem of Drinfeld and Simpson \cite{Drin-Simp}, we may find a set $\{p_i\}$ of closed points of $C$ and a finite separable extension $k'$ of $k$ that simultaneously trivializes the restriction of $\mathcal{G}$ over $\dot{C}:=C\setminus \{p_i\}$ and the fibers over $p_i$. Consequently we obtain the following Gysin triangle
$$
M(G) \otimes M(\dot{C}_{k'})\rightarrow M(\mathcal{G}_{k'})\rightarrow \bigoplus_{p_i} M(G_{k'})(n)[2n].
$$
Now \ref{ba} follows from Lemma \ref{ReuctiveGroup}.\\

\noindent
Now we prove \ref{jim}. Since $G$ is reductive, $Z(G)=Z(G)^\circ$ is a torus. Thus it suffices to prove the statement for the associated $G^{ad}$-bundle $\CG'$, see example \ref{TorusBundle}. Regarding Example \ref{ReductiveBundle}, part \ref{jim} follows from filtration \eqref{Filt1} and the motivic Leray-Hirsch theorem \ref{Leray Hirsch for Voevodsky motives}.\\
\noindent
Finally when $k=\BC$, Iyer \cite{Iyer} establishes the Chow-K\"unneth decomposition for a projective homogeneous fiber bundle (which is locally trivial with respect to the \'etale topology on the base). Note however that this is done only after passing to rational coefficients. Since $G^{ad}$-bundle $\CG'$ is locally trivial for the \'etale topology on $X$, so does the fiber bundle $\mathcal{D}_{\cF}$. The fiber $D_F$ is itself a projective homogeneous fibre bundle over a projective homogeneous variety which is determined by the face $F$ \cite[Proposition 2.26]{EvJo}. Thus we may deduce part \ref{dal} by applying \cite[Proposition 3.8]{Iyer} to all the fiber bundles $\mathcal{D}_\cF$ appearing in the filtration \eqref{Filt1}. \\

\end{proof}

At the end of this section, it may look worthy to state the corresponding facts in the K-ring $\textbf{K}_0(\textbf{Var}_k)$ (as well as the K-ring $\textbf{K}_0(\textbf{DM}_{gm}^{eff}(k))$).\\ 



\forget
{
\begin{remark}\label{Cor_K_Ring}
Recall that for a fibration $X\to Y$ with fiber $F$, which is locally trivial for Zariski topology, one has $[X]=[Y] . [F]$ (here $[.]$ denotes the corresponding class in $K_0(Var_k)$), according to Gillet and Soul\'e \cite[Proposition 3.2.2.5]{GS}. 
Now assume that $k$ is algebraically closed and let $\CG$ be a $G$-bundle over a smooth variety $X$. Furthermore, assume that either of the assumptions \ref{alef}, \ref{ba}, \ref{jim} or \ref{dal} of Theorem \ref{MGB} hold, then one can verify that the torsor relation $[M(\CG)]=[M(G)].[M(X)]$ holds in $\textbf{K}_0(\textbf{DM}_{gm}^{eff}(k)\otimes\BQ)$. Moreover in first three cases one even has $[\CG]=[G].[X]$, already in $\textbf{K}_0(\textbf{Var}_k)$.
\end{remark}
}
\noindent
Recall that for a fibration $X\to Y$ with fiber $F$, which is locally trivial for Zariski topology, one has $[X]=[Y] . [F]$ (here $[.]$ denotes the corresponding class in $K_0(Var_k)$), according to Gillet and Soul\'e \cite[Proposition 3.2.2.5]{GS}. 

\begin{proposition}\label{Prop_K_Ring}
Assume that $k$ is algebraically closed. Let $\CG$ be a $G$-bundle over a smooth variety $X$. Furthermore, assume that either of the assumptions \ref{alef}, \ref{ba}, \ref{jim} or \ref{dal} of Theorem \ref{MGB} hold, then the torsor relation $[M(\CG)]=[M(G)].[M(X)]$ holds in $\textbf{K}_0(\textbf{DM}_{gm}^{eff}(k)\otimes\BQ)$. Moreover in first three cases we have $[\CG]=[G].[X]$, already in $\textbf{K}_0(\textbf{Var}_k)$.
\end{proposition}

\begin{proof}
Assume that \ref{alef} (resp. \ref{ba}) holds, then we may argue by filtration \ref{FiltOverCellX} (resp.  the theorem of Drinfeld and Simpson). If \ref{jim} holds, then it follows from Lemma \ref{NestedFilt} and Theorem \ref{Leray Hirsch for Voevodsky motives}, or directly by Gillet and Soul\'e \cite[Proposition 3.2.2.5]{GS}. Finally, assuming \ref{dal}, we may conclude by Lemma \ref{NestedFilt} and \cite[Proposition~3.8]{Iyer}. 
\end{proof}

\begin{remark}\label{torsorrelation}
The above torsor relation $[M(\CG)]=[M(G)].[M(X)]$ was also proved in \cite[Appendix~A]{Beh-Dhi}. However they only treat the case that the characteristic of the ground field is $0$. Note further that they also use the assumption that the center of $G$ is connected, see proof of \cite[Theorem~A.9]{Beh-Dhi}, although they don't mention this precisely. 
\end{remark}

\def\@makecol{\ifvoid\footins \setbox\@outputbox\box\@cclv
   \else\setbox\@outputbox
     \vbox{\boxmaxdepth \maxdepth
     \unvbox\@cclv\vskip\skip\footins\footnoterule\unvbox\footins}\fi
  \xdef\@freelist{\@freelist\@midlist}\gdef\@midlist{}\@combinefloats
  \setbox\@outputbox\hbox{\vrule width\marginrulewidth
        \vbox to\@colht{\boxmaxdepth\maxdepth
         \@texttop\dimen128=\dp\@outputbox\unvbox\@outputbox
         \vskip-\dimen128\@textbottom}%
        \vrule width\marginrulewidth}%
     \global\maxdepth\@maxdepth}
\newdimen\marginrulewidth
\setlength{\marginrulewidth}{.1pt}
\makeatother

\setlength{\marginrulewidth}{0pt}

\vfill

\begin{minipage}[t]{0.5\linewidth}
\noindent
Esmail Arasteh Rad\\
Universit\"at M\"unster\\
Mathematisches Institut \\
erad@uni-muenster.de
\\[1mm]
\end{minipage}
\begin{minipage}[t]{0.45\linewidth}
\noindent
Somayeh Habibi\\
School of Mathematics,\\ 
Institute for Research in\\
 Fundamental Sciences
(IPM)\\
P.\ O.\ Box: 19395-5746, Tehran, Iran
shabibi@ipm.ir
\\[1mm]
\end{minipage}

{}


\begin{thebibliography}{10}


\bibitem{Ar-Ha} 
E.\ Arasteh Rad, S.\ Habibi: \emph{Local Models For The Moduli Stacks of Global $\FG$-Shtukas}, preprint.


\bibitem{AH_Local} E.~Arasteh Rad, U.~Hartl: \emph{Local $\BP$-shtukas and their relation to global $\FG$-shtukas}, Muenster J.~Math (2014); also available as \href{http://arxiv.org/abs/1302.6143}{arXiv:1302.6143}.


\bibitem{AH_Global} 
E.~Arasteh Rad, U.~Hartl: \emph{Uniformizing the moduli stacks of global $\FG$-Shtukas}, to appear in Int.\ Math.\ Res.\ Notices (2014); also available as \href{http://arxiv.org/abs/1302.6351}{arXiv:1302.6351}.




\bibitem{Beh-Dhi}  K.\ Behrend and A.\ Dhillon. \emph{On the Motive of the Stack of Bundles} \href{http://arxiv.org/abs/math/0512640}{arXiv:0512640}.


\bibitem{Big} 
S.\ Biglari, Motives of Reductive Groups, Thesis, Leipzig (2004).

\bibitem{Brion} 
M.\ Brion: \emph{The behaviour at infinity of the Bruhat decomposition}. \emph{Comment.\ Math.\ Helv.} 73,137-174 (1998).

\bibitem{B-T} 
F.\ Bruhat, J.\ Tits: Groupes r\'eductifs sur un corps local. Inst. Hautes \'Etudes Sci. Publ. Math. 41
, 5-251 (1972).

\bibitem{B-TII} 
F.\ Bruhat, J.\ Tits: Groupes r\'eductifs sur un corps local. II. Sch\'emas en groupes. Existence d\'une donn\'e radicielle valu\'ee. Inst. Hautes \'Etudes Sci. Publ. Math. 60 , 197-376 (1984).


\bibitem{Choud}
 U.\ Choudhury: \emph{Motives of Deligne-Mumford stacks}, Adv.\ Math.,\ 231(6):3094-3117 (2012). 



\bibitem{Col-Ful} 
A. \ Collino and W. \ Fulton, \emph{Intersection rings of spaces of triangles}, M\'em.
Soc.~Math.~France~(N.S.) 38, 75-117  (1989).

\bibitem{CGM} 
V.\ Chernousov, S.\ Gille, A.\ Merkurjev: \emph{Motivic decomposition of isotropic projective homogeneous varieties}, Duke Math. J. 126, no. 1, 137-159  (2005).


\bibitem{C-H}
A.\ Corti, M.\ Hanamura: \emph{Motivic decomposition and intersection Chow groups
I}, Duke Math.\ J.\ 103, 459-522 (2000).

 
\bibitem{Con-Pro} 
C.\ De Concini, C.\ and Procesi: \emph{Complete symmetric varieties}, Invariant theory (Montecatini,1982), volume 996 of Lecture Notes in Math., pages 1–44, Springer, Berlin (1983).


\bibitem{C-M}
M.\ de Cataldo, L.\ Migliorini: \emph{The Chow Motive of semismall resolutions},
Math.Res.Lett. 11, 151-170 (2004).

\bibitem{DeConciniProcesi} 
C.\ DeConcini and C.\ Procesi: \emph{Complete symmetric varieties}, Lecture Notes in Math. 996, Springer, 1-44 (1973).

\bibitem{Drin-Simp} 
V.\ Drinfeld and C.\ Simpson: \emph{B-structures on G-bundles and local triviality},
 Math.-Res.-Lett. 2, 823-829 (1995).
 
\bibitem{EvJo} 
S.\ Evens, B.\ Jones: \emph{On the Wonderful Compactification}, arXiv: 0801.0456 (2008).

\bibitem{GS} 
H.\ Gillet and C.\ Soul\'e: \emph{Descent, motives and K-theory}. J.\ reine angew.\ Math., 4 (1996).


\bibitem{Ha} 
T.\ Haines: \emph{Structure constants for Hecke and representation rings}, Int.\ Math.\ Res.\ Not.\ 39: 2103-2119 (2003).


\bibitem{H-K} 
A.\ Huber , B.\ Kahn: \emph{The slice filtration and mixed Tate motives}, \emph{Compos.\ Math.}, 142(4): 907-936 (2006).


\bibitem{Iyer}
J.~Iyer, \emph{Absolute Chow-K\"unneth decomposition for rational homogeneous bundles and for log homogeneous
varieties}. Michigan Math.~J., 60(1):79-91 (2011).


\bibitem{BKahn} 
B.\ Kahn : \emph{Motivic Cohomology of smooth geometrically cellular varieties}, Algebraic K-theory,
Seattle, 1997, Proceedings of Symposia in Pure Mathematics, vol. 67 (American Mathematical Society, Providence, RI, 149-174)  (1999). 

\bibitem{Kar} 
N.\ A.\ Karpenko: \emph{Cohomology of relative cellular spaces and of isotropic flag varieties}, Algebra i Analiz 12 , no. 1, 3-69 (2000).

\bibitem{Kiritchenko}  
V.\ Kiritchenko: \emph{On intersection indices of subvarieties in reductive groups} Moscow Mathematical Journal, 7 no.3, 489-505, (2007).

\bibitem{Ko}
B.\ Kock: \emph{Chow motif and higher Chow theory of G/P}, Manuscripta Math.\ 70 , 363-372 (1991).


\bibitem{LAM} 
T.\ Y.\ Lam: \emph{Serre's Problem on Projective modules}, Springer Monographs in Mathematics (2006).

\bibitem{Las-Sor} 
Y.\ Laszlo and C.\ Sorger: \emph{The line bundles on the moduli of parabolic G-bundles
over curves and their sections}. \emph{Ann. Sci. Ecole Norm. Sup. (4)}, 30(4):499-525 (1997). 


\bibitem{MVW} 
C.\ Mazza, V.\ Voevodsky, C.\ A.\ Weibel. \emph{Lecture notes on motivic cohomology}, Clay mathematics monographs, v.2. (2006).

\bibitem{Ngo-Polo} 
B.C.\ Ng\^o, P.\ Polo: \emph{R\'esolutions de Demazure affines et formule de Casselman-Shalika g\'eom\'etrique}, J.\ Algebraic Geom.\ {\bfseries 10} (2001), no. 3, 515--547; also available at \href{http://www.math.uchicago.edu/~ngo/ngo-polo.pdf}{http:/\!/www.math.uchicago.edu/$\sim$ngo/}. 


\bibitem{P-R}
G.\ Pappas, M.\ Rapoport: \emph{Twisted loop groups and their affine flag varieties}, Adv.\ Math.\ 219 , no.\ 1, 118-198 (2008).


\bibitem{PRS}
G.\ Pappas, M.\ Rapoport and B.\ Smithling: \emph{Local models of Shimura varieties}, I.
Geometry and combinatorics, preprint, arXiv: 1011.5551 (2010).


\bibitem{Raghunathan1}
 M.\ S.\ Raghunathan: \emph{Principal bundles on affine space and bundles on the projective line} \emph{Mathematische Annalen} vol. 285, no.\ 2, 309-332 (1989).

\bibitem{Raghunathan2} 
M.\ S.\ Raghunathan: \emph{Principal bundles on affine space}, in C. P. Ramanujam—a tribute, pp. 187-206, Tata Inst. Fund. Res. Studies in Math. 8 (1978).

\bibitem{Renner} 
L.\ E.\ Renner: \emph{An Explicit Cell Decomposition of the Wonderful Compactification of a Semisimple Algebraic Group}, \emph{Canadian Mathematical Bulletin}, 46 (1): 140-148 (2003).


\bibitem{Ri}
T.\ Richarz: \emph{Schubert varieties in twisted affine flag varieties and local models}, Journal of Algebra 375, 121-147 (2013).



\bibitem{Sch}
 J.\ Scholbach: \emph{Mixed Artin-Tate motives over number rings}, \emph{Journal of Pure and Applied Algebra} 215, 2106–2118 (2011).


\bibitem{Tima} 
D.\ Timashev, \emph{Equivariant compactifications of reductive groups}, Sb. Math. 194, no. 3-4, 589-616 (2003).

\bibitem{Toen}
 B.~ Toen: \emph{On motives of Deligne-Mumford stacks}, Int.\ Math.\ Res.\ Notices (17): 909-928, (2000).


\bibitem{Var} 
Y.\ Varshavsky: \emph{Moduli spaces of principal $F$-bundles}, Selecta Math.\ (N.S.) {\bfseries 10} (2004),  no.\ 1, 131--166; also available as \href{http://arxiv.org/abs/math/0205130}{arXiv:math/0205130}.


\bibitem{VVa} 
V. Voevodsky, \emph{Motivic cohomology groups are isomorphic to higher Chow groups in any characteristic}, Int. Math. Res. Notices, no. 7, 351-355 (2002).


\bibitem{VV}
V.\ Voevodsly, A.\ Suslin, E.M.\ Friedlander: \emph{Cyles, Transfers, and Motivic Homology Theories}, Princeton university press (2000).


\bibitem{VVI} 
V.\ Voevodsky: \emph{Motives over simplicial schemes}: J. K-Theory 5,
no. 1, 1-38 (2010).

\bibitem{VVII} 
V.\ Voevodsky: \emph{Homology of Schemes}, Selecta Mathematica, New Series, 2(1):111-153 (1996).




\end{thebibliography}
\end{document}